\documentclass[11pt,a4paper]{mathscan}
\usepackage{amssymb}
\usepackage{amscd}        
\usepackage{graphics}

\usepackage{color}

\usepackage[all]{xy}
\usepackage{amsthm}   
\usepackage{natbib}
\usepackage{mathrsfs}
\numberwithin{equation}{section}
\theoremstyle{plain}
\newtheorem{theorem}{Theorem}[section]
\newtheorem{lemma}[theorem]{Lemma}
\newtheorem{corollary}[theorem]{Corollary}

\theoremstyle{definition}
\newtheorem{definition}[theorem]{Definition}
\newtheorem{example}[theorem]{Example}
\theoremstyle{remark}
\newtheorem{remark}[theorem]{Remark}
\newcommand{\RMOD}{\ensuremath{{R\text{-}\!\mathcal{M}od}}} 

\newcommand{\Hom}{\operatorname{Hom}}

\newcommand{\X}{\ensuremath{X}}
\newcommand{\Y}{\ensuremath{Y}}

\newcommand{\coclosed}[2]{\ensuremath{\xymatrix@1{ #1 \, \ar@{^{(}->}[r]^-{cc} &
#2}}}
\newcommand{\cosmall}[3]{\ensuremath{\xymatrix@1{ #1 \,
\ar@{^{(}->}[r]^-{cs}_-{#2} & #3}}}

\newcommand{\ShortExactSequence}[5]{\ensuremath{\xymatrix@1{ 0 \ar[r] &  #1 \ar[r]^-{#2} & #3 \ar[r]^-{#4} &  #5
\ar[r] & 0 }}}

 \newcommand{\To}{\longrightarrow}   

\begin{document}
\title{Transfinite tree quivers and their representations}


\author{E. Enochs, S. Estrada and S. \"{O}zdemir}

\address{Department of Mathematics, University of Kentucky,
Lexington, Kentucky 40506-0027,
U.S.A.}\email{enochs@ms.uky.edu}\address[]{Departamento de Matem\'{a}tica
Aplicada, Universidad de Murcia, Murcia 30100,
SPAIN.}\email{sestrada@um.es}\
\address[]{Dokuz Eyl\"{u}l \"{u}n\.{i}vers\.{i}tes\.{i}, Fen Fak\"{u}ltes\.{i},
Matemat\.{i}k B\"{o}l\"{u}m\"{u}, \.{I}zm\.{i}r, TURKEY.}
\email{salahattin.ozdemir@deu.edu.tr}

\keywords{transfinite tree, source injective representation quiver,
indecomposable injective}

\subjclass[2000]{16G20, 18A40}

\thanks{The second author has been partially supported by DGI
MTM2008-03339, by the Fundaci\'on Seneca
07552-GERM and by the Junta de Andaluc\'{\i}a,
Consejer\'{\i}a de Econom\'{\i}a,
Innovaci\'on y Ciencia and FEDER funds. The third author has been
supported by The Council of Higher Education (Y\"{O}K) and by The Scientific
Technological Research Council of Turkey (T\"{U}B\.{I}TAK)}

\renewcommand{\theenumi}{\arabic{enumi}}
\renewcommand{\labelenumi}{\emph{(\theenumi)}}

\begin{abstract}
The idea of ``vertex at the infinity'' naturally appears when studying
indecomposable injective representations of tree quivers. In this paper we
formalize this behavior and find the structure of all the indecomposable
injective representations of a tree quiver of size an arbitrary cardinal
$\kappa$. As a consequence the structure of injective representations of
noetherian $\kappa$-trees is completely determined. In the second part we will
consider the problem whether arbitrary trees are source injective representation
quivers or not.
\end{abstract}

\maketitle
\section{Introduction}
The classical representation theory of quivers motivated by Gabriel's work (\citet{Gabriel:UnzerlegbareDarstellungen}) involved finite quivers and assumed that the base ring is algebraically closed field and that all vector spaces involved were finite dimensional. Recently, representations by modules of more general (possibly infinite) quivers over any ring have been studied. The aim of this paper is to continue with the program initiated in \citet{Enochs-Herzog:HomotopyOfQuiverMorphismsWithApplicationsToRepresentations} and continued in \citet{Enochs-et.al:NoetherianQuivers}, \citet{Enochs-et.al:FlatCoversofRepresentationsOfTheQuiverA}, \citet{Enochs-et.al:FlatCoversAndFlatRepresentationsOfOuivers}, \citet{Enochs-Estrada:ProjectiveRepresentationsOfQuivers}, \citet{Enochs-et.al:InjectiveRepresentationsOfQuivers} and \citet{Estrada-Ozdemir:RelHomAlgInCatOfQuivers} to develop new techniques on the study of these more general representations.

Our main concern on this paper is to study injective representations of transfinite tree quivers. These quivers have been recently considered by Rump in \citet{Rump:FlatCoversInAbelianAndInNon-abelianCategories} in his study of the existence of flat covers
on certain non-necessarily Abelian categories. Transfinite tree quivers also
appear naturally
when studying indecomposable injective representations of tree quivers. Namely,
in \citet[Section 3]{Enochs-et.al:InjectiveRepresentationsOfQuivers} it is
proved that the indecomposable injective representations of a tree
are in one-to-one correspondence with both finite and infinite sequences of
modules of the form $E\stackrel{id}{\longrightarrow}
E\stackrel{id}{\longrightarrow}\cdots\stackrel{id}{\longrightarrow} E$ and
$E\stackrel{id}{\longrightarrow} E\stackrel{id}{\longrightarrow}
E\stackrel{id}{\longrightarrow} \cdots $ (where
$E$ is an indecomposable injective $R$-module and $id$ is the identity map). Each one of these sequences
corresponds to the different finite or infinite paths of the tree starting in
the root. For instance for the tree quiver $$A^{\infty}=v_1\longrightarrow
v_2\longrightarrow \ldots
\longrightarrow v_n\longrightarrow\ldots,$$
the indecomposable injective representations are of the form
$$E_n=E\stackrel{id}{\longrightarrow} E\stackrel{id}{\longrightarrow}
E\stackrel{id}{\longrightarrow}\ldots\stackrel{id}{\longrightarrow}
E\longrightarrow 0\longrightarrow 0\longrightarrow \ldots$$ and
$$E^{\infty}=E\stackrel{id}{\longrightarrow} E\stackrel{id}{\longrightarrow} E
\stackrel{id}{\longrightarrow} E\stackrel{id}{\longrightarrow}
\ldots .$$
Hence when studying indecomposable injective representations of a tree one has
to consider simultaneously both the vertices of the quiver and the ``vertices at
the inifinity''. In the present paper we make more precise the statement
``adding vertices at the infinity'' by introducing the notion of the completion
of a tree, and studying in Section \ref{sec:transfinite-trees-and-its-completion} the category of cocontinuous
representations of such completed trees. As an application of this,
we characterize in Theorem \ref{thm:characterization-of-indec.inj.rep-transfinite-trees} the indecomposable injective representations
of trees of any size in terms of its completion.

In the second part of the paper we will focus on the local properties of the
injective representations of trees.The class of
source injective representation quivers has been introduced in
\citet[Definition 2.2]{Enochs-et.al:InjectiveRepresentationsOfQuivers}. More
precisely, a quiver $Q$ is source
injective representation quiver whenever the injective representations $X$ of
$(Q, \RMOD)$ are characterized in terms of the following two local properties
(see Section \ref{sec:preliminaries} for unexplained terminology):
\begin{enumerate}
\item[i) ] ${\mathcal{X}}(v)$ is an injective $R$-module, for any vertex $v$ of
$Q$.
\item[ii) ] For any vertex $v$, the morphism
$${\mathcal{X}}(v)\longrightarrow \prod_{s(a)=v}{\mathcal{X}}(t(a))$$ induced by
${\mathcal{X}}(v)\longrightarrow {\mathcal{X}}(t(a))$ is a splitting
epimorphism.
\end{enumerate}
As it is shown in \citet[Theorem 4.2]
{Enochs-et.al:InjectiveRepresentationsOfQuivers} this important class of
quivers includes finite quivers and more generally
right rooted quivers, but do not include cyclic quivers (see \citet[Example
1]{Enochs-et.al:InjectiveRepresentationsOfQuivers}). Furthermore,
concerning to trees, it is shown in \citet[Corollary
5.5]{Enochs-et.al:InjectiveRepresentationsOfQuivers} that (possibly infinite)
barren trees are
examples of source injective representation quivers. So implicitly the question
whether all trees are source injective representation quivers or not was arisen.
We solve this question in the negative in Section \ref{sec:non-source-inj-rep-trees} by showing that non-barren
trees such that the set $\{t(a):\ s(a)=v\}$ is finite, for each
vertex $v$ of $Q$, are not source injective representation quivers.

\section{Preliminaries}\label{sec:preliminaries}
All rings considered in this paper will be associative with identity
and, unless otherwise specified, not necessarily commutative. The
letter $R$ will usually denote a ring. All modules are left
unitary $R$-modules. $\RMOD$ will denote the category of left $R$-modules.  We refer to
\citet{Simson-et.al:ElementsOfTheRepresentationTheoryOfAssociativeAlgebras}
for any undefined notion used in the text.

A \emph{quiver} is a directed graph whose edges are called arrows.
As usual we denote a quiver by $Q$ understanding that $Q= (V, E)$
where $V$ is the set of vertices and $E$ the set of arrows. An arrow
of a quiver from a vertex $v_1$ to a vertex $v_2$ is denoted by $a:
v_1 \to v_2$. In this case we write $s(a) = v_1$ the initial
(starting) vertex and $t(a) = v_2$ the terminal (ending) vertex. A
\emph{finite path} $p$ of a quiver $Q$ is a sequence of arrows $a_n \cdots
a_2 a_1$ with $t(a_i) = s(a_{i+1})$ for all $i=1, 2, \ldots, n-1$. Thus $s(p) = s(a_1)$ and $t(p)
= t(a_n)$. Two paths $p$ and $q$ can be composed, getting another
path $qp$ (or $pq$) whenever $t(p) = s(q)$ ($t(q) = s(p)$).

A quiver $Q$ may be thought of as a category in which the objects are
the vertices of $Q$ and the morphisms are the paths of $Q$. 
The vertices can be considered as the identities of $Q$, that is, a vertex $v$
of $Q$ is a trivial path where $s(v)=t(v) = v$. A (right) \emph{rooted} quiver
is a quiver having no path of the form $\bullet \to \bullet \to \ldots $ (see
\citet{Enochs-et.al:FlatCoversAndFlatRepresentationsOfOuivers}). A \emph{tree}
is a quiver $T$  having a vertex $v$ such that for another vertex $w$ of $T$,
there exists a unique path $p$ such that $s(p) = v$ and $t(p) = w$. Such a
vertex is called the root of the tree $T$. The (left) \emph{path space} of a
quiver $Q$, denoted by $P(Q)$, is the quiver whose vertices are the paths $p$ of
$Q$ and whose arrows are the pairs $(p, ap):p\to ap$ such that $ap$ is defined
(i.e. $s(a) = t(p)$). For each vertex $v\in V$, $P(Q)_v$ is a \emph{subtree} of
$P(Q)$ containing all paths of $Q$ with initial vertex $v$. A tree with a root
$v$ is said to be  \emph{barren} if the number of vertices $n_i$ of the $i$th
state of $T$ is finite for every natural number $i$ and the sequence of positive
natural numbers $n_1, n_2, \ldots$ stabilizes (see
\citet{Enochs-et.al:NoetherianQuivers}). For example, the tree $$
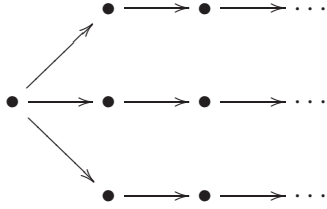
\xymatrix{  & 
\bullet \ar[r] &  \bullet \ar[r] & \cdots \\
\bullet \ar[r]  \ar[ur] \ar[dr]  &  \bullet \ar[r] & \bullet \ar[r]
& \cdots
\\ &  \bullet \ar[r] & \bullet \ar[r] & \cdots }$$ is barren.

A representation by modules $\X$ of a given quiver $Q$ is a functor
$\X: Q \To \RMOD$. Such a representation is determined by giving a
module $\X(v)$ to each vertex $v$ of $Q$ and a homomorphism $\X(a):
\X(v_1) \to \X(v_2)$ to each arrow $a: v_1 \to v_2$ of $Q$. A
morphism $\eta$ between two  representations $\X$ and $\Y$ is a
natural transformation, so it will be a family  $\{\eta_v : X(v) \to Y(v)\}_{v\in V}$ such that
$\Y(a)\circ \eta_{v_1} = \eta_{v_2} \circ \X(a)$ for any arrow
$a:v_1 \to v_2$ of $Q$. Thus the representations of a quiver $Q$ by
modules over a ring $R$ form a category, denoted by $(Q, \RMOD)$. This is a Grothendieck category with enough projectives and injectives.

The category $(Q, \RMOD)$ is equivalent to the category of modules over the \emph{path ring} $RQ$, where $RQ$ is defined as the free left $R$-module whose
base are the paths of $Q$, and where the multiplication is the
obvious composition between two paths. $RQ$ is a ring with enough idempotents and in general it does not have an identity (unless the set $V$ is finite).

For a given quiver $Q$, one can define a family of projective
generators from an adjoint situation as it is shown  in
\citet{Mitchell:RingsWithSeveralObjects}. For every vertex $v\in V$
and the embedding morphism $\{v\} \subseteq Q$, the family $\{S_v(R)
: v\in V\}$ is a family of projective generators for the category of representations $(Q, \RMOD)$ where the
functor $S_v : \RMOD \to (Q, \RMOD)$ is defined in \citet[\S
28]{Mitchell:RingsWithSeveralObjects} as
\[
S_v (M)(w) = \bigoplus_{Q(v, w)} M
\]
 and for any arrow $a:w_1 \to w_2$ of $Q$,
 \[
 S_v(M)(a): \bigoplus_{Q(v, w_1)} M \to \bigoplus_{Q(v, w_2)} M
\]
is given by $\displaystyle S_v(M)(a)= \bigoplus_{Q(v, w_2)} id_{Q(v, w_2)}$,
 where $Q(v, w)$ is the set of paths of $Q$ starting at $v$
and ending at $w$. Then $S_v$ is a left adjoint functor of the
evaluation functor $T_v: (Q, \RMOD) \to (\{v\}, \RMOD)\cong \RMOD$ given by $T_v (X) =
X(v)$ for any representation $X\in (Q, \RMOD)$, and so given representations $M$ of $(\{v\}, \RMOD)$ (i.e. an $R$-module $M$) and $X$ of $(Q, \RMOD)$ there is a natural isomorphism: $$\Hom_{(Q, \RMOD)} (S_v (M), X)  \cong \Hom_{\RMOD} (M, X(v)).$$

Similarly, we can find a family of injective cogenerators from an adjoint situation (see \citet{Enochs-Herzog:HomotopyOfQuiverMorphismsWithApplicationsToRepresentations}). For every vertex $v \in V$ and the embedding morphism $\{v\} \subseteq Q$, the family $\{e_*^v (E) : v \in V \}$ is a family of injective cogenerators of $(Q, \RMOD)$, whenever $E$ is an injective cogenerator of $\RMOD$. The functor $e_*^v : \RMOD \to (Q, \RMOD)$ is defined in
\citet[\S 4]{Enochs-Herzog:HomotopyOfQuiverMorphismsWithApplicationsToRepresentations} as $$ e_*^v (M)(w) = \prod_{Q(w, v)}M,$$  and if $a:w_1 \to w_2$ is an arrow of $Q$, then $$e_*^v(M)(a): \prod_{Q(w_1, v)}M \to \prod_{Q(w_2, v)}M$$ is given by $\displaystyle e_*^v(M)(a) = \prod_{Q(w_2, v)a }id_{Q(w_2, v)}$. Then by \citet[Theorem 4.1]{Enochs-Herzog:HomotopyOfQuiverMorphismsWithApplicationsToRepresentations}, $e_*^v$ is the \emph{right} adjoint functor of  $T_v$, and so given representations $M$ of $(\{v\}, \RMOD)$ (i.e. an $R$-module $M$) and $X$ of $(Q, \RMOD)$ there is a natural isomorphism:
$$\Hom_{(Q, \RMOD)} (X, e_*^v(M)) \cong \Hom_{\RMOD} (X(v), M).$$


\section{Transfinite trees and their representations}\label{sec:transfinite-trees-and-its-completion}

A \emph{well-ordered} set is a totally ordered set satisfying the descending chain condition. The prototype of such a set is $\mathbb{N} = \{0, 1, 2, \ldots\}$. So, a tree $T$ can be regarded as a partially ordered set, where $u \leq v$ if there is a path from $u$ to $v$ ($u, v \in T$). Such a tree can be thought of as a generalization of $\mathbb{N}$ as a partially ordered set. The following is a way to define trees which correspond to ordinal numbers in general (where $\mathbb{N}$ has the ordinal number $\omega$).

\begin{definition}\label{defn:transfiniteTREE}
By a \emph{transfinite tree} we mean a partially ordered set $T$ satisfying
\begin{itemize}
  \item[(i)] the descending chain condition, that is, if $v_0 \geq v_1 \geq \cdots$ then for some $n_0$, we have $v_{n_0} = v_n$ for all $n\geq n_0$, $v_i \in T$;
  \item[(ii)] for any $v\in T$, the set of $w \leq v$ is totally ordered, that is, if $w, w' \leq v$ then either $w\leq w'$ or $w' \leq w$;
  \item[(iii)] there is a least element $v\in T$, that is, $v\leq v'$ for all $v' \in T$.
\end{itemize}
\end{definition}

Now, we partition the transfinite tree $T$ as follows (the partition will be indexed by the ordinal numbers):

Let $T_0 = \{v\}$ where $v$ is the least element of $T$. Let $T_1$ be the set of $w\in T$ such that $v\neq w$ and such that $v\leq u \leq w$ implies that $v=u$ or $u=w$ or equivalently, the cardinality of the set of $u\leq w$ is $2$. Then define $T_2$ to be the set of $w\in T$ such that the cardinality of the set $u\leq w$ is $3$. Similarly, we define $T_3, T_4, T_5, \ldots$. Therefore, we define $T_{\omega}$ to be the set of $w\in T$ such that if $u < w$ then $u\in T_n$ for some $n \in \mathbb{N}$. So, for an ordinal number $\alpha$, we define $T_{\alpha}$ by transfinite induction: \\
So having defined $T_{\alpha}$ for $\alpha < \beta$, we must define $T_{\beta}$. If $\beta$ is not a limit ordinal, then $\beta = \alpha + 1$ for some $\alpha$. Then we let $T_{\alpha + 1}$ consists of $w\in T$ such that $v < w$ for some $v\in T_{\alpha}$, where $v \leq u \leq w$ implies that $v=u$ or $u=w$. If $\beta$ is a limit ordinal, we let $w\in T_{\beta}$ if whenever $v < w$ we have $v\in T_{\alpha}$ for some $\alpha < \beta$ and if, for every $\alpha' < \beta$, there is a $v < w$ with  $v\in T_{\alpha}$ where $\alpha' \leq \alpha < \beta$.

\begin{remark}
If $T_{\omega} = \emptyset$ then $T$ is a tree in the usual sense, where there is an arrow from $u$ to $v$ if $u < v$, and  $u\leq w \leq v$ implies that $u=w$ or $w = v$.
\end{remark}

The following is an example of a transfinite tree which is not a (usual) tree.

\begin{example}\label{exm:transfiniteTREE-not-usual-TREE}
The tree $T \equiv v_0 \to v_1 \to \cdots \to v_n \to \cdots v_{\omega}$ (where $\omega$ is the first limit ordinal) is not a usual tree. Indeed, $T_{\omega} \neq \emptyset$ since $v_{\omega} \in T_{\omega}$ (as for all $u < v_{\omega}$, $u\in T_n$ for some $n\in \mathbb{N}$).
\end{example}

\begin{definition}
A tree $T$ is said to be \emph{complete} if every chain (i.e. totally ordered subset) of $T$ has a least upper bound in $T$.
\end{definition}

\begin{remark}
Every tree $T$ has a completion $T \subseteq \overline{T}$. This means that $\overline{T}$ is a complete tree, $T$ is a subtree of $\overline{T}$ such that if $v\in T$ and $u\leq v$ for some $u\in \overline{T}$, then $u\in T$, and such that every $v\in \overline{T}$ is the least upper bound of a chain in $T$.
\end{remark}

Notice that, if $T$ is a transfinite tree then we can make $T$ into a category, where $\Hom (u, v)$ is empty if $u\nleq v$. Then, for $u, v, w \in T$,
$$\Hom(v, w) \times \Hom(u, v) \To \Hom(u, w)$$ is defined in the obvious manner. Thus a representation $X$ of $T$ in $\RMOD$ is just a functor $X:T \to \RMOD$.

We point out that all representations of $(T, \RMOD)$, where $T$ is a
transfinite tree, that we consider are cocontinuous. 
This is because we want to generalize usual representations of (usual) trees (or
in general of quivers), that is, we know that if $Q$ is any quiver and $X$ is a
representation of $(Q, \RMOD)$ then, for a finite path $v_{i_1} \to v_{i_2} \to
\cdots \to v_{i_n}$ on $Q$, we have trivially that $$v_{i_n} = \sup \,\{v_{ij}
\mid j \leq n\}\quad \text{ and } \quad X(v_{i_n}) = {\lim_{\rightarrow}}_{j\leq
n} X(v_{ij}). \qquad (*)$$ But if we admit \emph{infinite} paths (as it happens
in transfinite trees), then we want our representations to satisfy the same
property $(*)$, that is, if $u = \sup \, \{v_{\alpha} \mid \alpha < \gamma\}$
(and so $u\in T_{\gamma}$), where $\gamma$ is an ordinal number, then $$X(u) =
{\lim_{\rightarrow}}_{\alpha < \gamma} X(v_{\alpha}).$$ Such $X$ are called
\emph{cocontinuous representations}. So our representations of transfinite trees
will be \emph{cocontinuous}. Of course, if $T$ is complete, then all such
supremums always exist, but if not, then we can consider the completion
$\overline{T}$ of $T$  because the categories  $(T, \RMOD)$ and $(\overline{T},
\RMOD)$ are easily shown to be equivalent.

It is then clear that the indecomposable injective
representations of $(T, \RMOD)$ are in  $1$-$1$ correspondence with the
indecomposable injective representations of $(\overline{T},
\RMOD)$.

\section{Indecomposable injective representations of transfinite trees}\label{sec:Indec.inj.rep.of-transfinite-trees}

\begin{lemma}
If $T_{\alpha} = \emptyset$ for any $\alpha$, then $T_{\beta} = \emptyset$ when $\alpha \leq \beta$.
\end{lemma}

\begin{proof}
Suppose on the contrary that
$T_{\beta} \neq \emptyset$. Then there exists a $w \in T_{\beta}$, and so there is a $v < w$ with $v \in T_{\alpha}$ for some
$\alpha < \beta$ (by definition of $T_{\beta}$) which is impossible (since $T_{\alpha} = \emptyset$).
\end{proof}

Notice that since $e_*^v (E)$ is a right adjoint functor of the evaluation functor $T_v$ as we have pointed out at the end of Section \ref{sec:preliminaries}, we have that if $E$ is an indecomposable injective $R$-module, then $e_*^v (E)$ is also  an indecomposable injective representation.
\begin{theorem}\label{thm:characterization-of-indec.inj.rep-transfinite-trees}
If $X$ is an indecomposable injective representation of $(\overline{T}, \RMOD)$ where $\overline{T}$ is a completion of the tree $T$, then $X \cong e_*^v (E)$ for some vertex $v\in \overline{T}$ and some indecomposable injective $R$-module $E$.
\end{theorem}

\begin{proof}
If we consider $T'_{\alpha}$'s which have been defined after Definition \ref{defn:transfiniteTREE}, then we see that they are not in a  chain. So let us define $\displaystyle T_{\gamma} = \bigcup_{\alpha \leq \gamma} T'_{\alpha}$ in order that the new $T_{\gamma}$'s are in a chain. Now, since $T$ is a set, $T = T_{\lambda}$ for some ordinal number $\lambda$ (and so $T_0 \subset T_1 \subset T_2 \subset \cdots T_{\omega} \subset \cdots \subset T_{\lambda} = T$).

\begin{enumerate}
\item[(i)] Firstly, we will prove by transfinite induction that if, for each $\alpha \leq \lambda$, there exists $v_{\alpha} \in T_{\alpha}$ such that $X(v_{\alpha}) \neq 0$, then $X\mid_{T_{\alpha}} \cong e_*^{v_{\alpha}} (E)$ (where $E = X(v_0))$. For $\alpha = 0$, we have trivially that $X\mid_{T_0} \cong e_*^{v_0} (E)$ (where $E = X(v_0))$. Assume that $\alpha$ has a successor and that $X\mid_{T_{\alpha}} \cong e_*^{v_{\alpha}} (E)$. We want to show that $X\mid_{T_{\alpha+1}} \cong e_*^{v_{\alpha+1}} (E)$. Since $X$ is an injective representation, we have, by
    \citet[Proposition 2.1]{Enochs-et.al:InjectiveRepresentationsOfQuivers}, a splitting epimorphism $$f: X(v_{\alpha}) \To X(v_{\alpha + 1}) \oplus \prod_{v_{\alpha}<u , u\neq v_{\alpha+1}} X(u) \quad (\text{ where } X(v_{\alpha+1}) \neq 0).$$
If $f$ is an isomorphism, then define a representation $Y'$ such that $Y' \mid_{T_{\alpha}} = X \mid_{T_{\alpha}}$ and, for $v\in T_{\beta}$ when $\alpha + 1 \leq \beta$, $Y' (v) = X(v)$ if the unique path from $v_{\alpha}$ to $v$ goes through $v_{\alpha+1}$ and $Y'(v)=0$ otherwise. Then let us define another representation $Y''$ such that $Y'' \mid_{T_{\alpha}} = 0$ and, for $v\in T_{\beta}$ when $\alpha +1 \leq \beta$, $Y'' (v) = 0$ if the unique path from $v_{\alpha}$ to $v$ goes through $v_{\alpha+1}$ and $Y'' (v) = X(v)$ otherwise. In fact, $Y'$ and $Y''$ are subrepresentations of $X$ and thus $X = Y' \oplus Y''$. So by  the indecomposability of $X$ we get $X = Y'$ since $Y' (v_{\alpha+1}) = X (v_{\alpha+1}) \neq 0$. In particular, $X \mid_{T_{\alpha+1}} \cong e_*^{v_{\alpha +1}} (E)$.

So now suppose that $f$ is not an isomorphism, and let $K$ be the kernel of $f$. Then define a representation $Y$ such that $Y(v)= K$ if $v \in T_{\gamma}$, $\gamma < \alpha+1$ and $Y(v)=0$ otherwise (then of course $Y(v_{\alpha+1}) = 0$). In fact, $Y$ is an injective subrepresentation of $X$ (since $K$ is an injective module as $f$ is splitting), and so a direct summand of $X$. But since $X$ is indecomposable and $Y\neq 0$ (as $K \neq 0$), we obtain $X = Y$. This is impossible because $X(v_{\alpha +1}) \neq 0$, but $Y(v_{\alpha +1})=0$.

Now assume that $\gamma \leq \lambda$ is a limit ordinal and that $X \mid_{T_{\alpha}} \cong e_*^{v_{\alpha}} (E)$ for all
$\alpha < \gamma$. We show that $X \mid_{T_{\gamma}} \cong e_*^{v_{\gamma}} (E)$. Let $v_{\gamma} = \sup \{v_{\alpha} \mid \alpha < \gamma \}$. Then
$$X (v_{\gamma}) = {\lim_{\rightarrow}}_{\alpha < \gamma} X (v_{\alpha}) = {\lim_{\rightarrow}}_{\alpha < \gamma} E = E $$
 and if $v_{\gamma} \neq u \in T_{\gamma}$, then
$$X(u)= {\lim_{\rightarrow}}_{\alpha < \gamma} X \mid_{T_{\alpha}} (u) = {\lim_{\rightarrow}}_{\alpha < \gamma} 0 = 0 .$$
So $X \mid_{T_{\gamma}} \cong e_*^{v_{\gamma}} (E)$.
  \item[(ii)] On the contrary to (i), let us assume that there exists $\alpha^* \leq \lambda$ such that $X(v) = 0$, for all $v\in T_{\alpha^*}$ and $\alpha^*$ is the smallest ordinal with this property. Then $\alpha^*$ cannot be a limit ordinal. Because, if $\alpha^*$ is a limit ordinal then  $X \mid_{T_{\beta}} \cong e_*^{v_{\beta}} (E)$ for all $\beta < \alpha^*$ by (i), but then $v_{\alpha^*} = \sup \{v_{\beta} \mid \beta < \alpha^*\} \in T_{\alpha^*}$ and $$X(v_{\alpha^*}) = {\lim_{\rightarrow}}_{\beta < \alpha^*} X \mid_{T_{\beta}} (v_{\beta}) = {\lim_{\rightarrow}}_{\beta < \alpha^*} E = E \neq 0 $$ which contradicts with our assumption. So $\alpha^*$ is a successor ordinal. Let $\alpha^* = \mu + 1$. Thus
      $X \mid_{T_{\mu}} \cong  e_*^{v_{\mu}}(E)$ by (i) (since $X(v_{\mu}) \neq 0$), and  $X(v_{\beta}) = 0$ for all $v_{\beta} \in T_{\beta}$, when $\beta \geq \mu + 1$ (since, by assumption, $X(v)=0$ for any $v\in T_{\alpha^*}$). Hence $X \cong e_*^{v_{\mu}}(E)$ in $T$.
\end{enumerate}
\end{proof}

The following definition is based on the characterization of locally noetherian categories of representations of quivers provided in \citet[Theorem 2.6]{Enochs-et.al:NoetherianQuivers}.
\begin{definition}
A transfinite tree $T$ is called \emph{noetherian} if $(T, \RMOD)$ is locally
noetherian for every left noetherian ring $R$, that is, $P(Q)_v$ is barren for
all $v\in T$ and $R$ is left noetherian.
\end{definition}

\begin{example}
Consider the tree $T$ as in Example \ref{exm:transfiniteTREE-not-usual-TREE} and the category of (cocontinuous) representations
$(T, \RMOD)$, where $R$ is a left noetherian ring. Then the (left) path spaces of $T$ are:
$$P(Q)_{v_0} \equiv v_0 \to v_1 \to \cdots \quad(\text{ not including } v_{\omega}),$$
$$P(Q)_{v_1} \equiv v_1 \to v_2 \to \cdots \quad(\text{ not including } v_{\omega}),$$
$$\vdots \hspace{5cm}$$
$$P(Q)_{v_{\omega}} \equiv v_{\omega} .\hspace{5cm}\qquad$$
 So  all of them are barren. Thus $T$ is noetherian.
\end{example}

Since $(T, \RMOD)$ is a Grothendieck category we obtain the following corollary.
\begin{corollary}(\textbf{Matlis Theorem for noetherian transfinite trees})\\
Let $T$ be a noetherian transfinite tree. Then any injective representation of $(T, \RMOD)$ is, uniquely up to isomorphism, the direct sum of the indecomposable injectives $e_*^v (E)$ for some $v\in T$, where $E$ is an indecomposable injective $R$-module.
\end{corollary}

\section{Non-source injective representation
trees}\label{sec:non-source-inj-rep-trees}
As we have pointed out in the introduction, infinite barren trees are source injective representation quivers. So, in this section, we show that non-barren trees which satisfy some condition on their vertices are not source injective representation quivers.

\begin{definition}\citep[Definition
2.2]{Enochs-et.al:InjectiveRepresentationsOfQuivers}

A quiver $Q$ is called \emph{a source injective
representation quiver} if, for
any ring $R$, any \emph{injective} representation $\X$ of $(Q,
\RMOD)$ can be characterized in terms of the following conditions:

\begin{itemize}
    \item[(i)]$\X (v)$ is injective $R$-module, for any vertex $v$
    of $Q$.
    \item[(ii)] For any vertex $v$ the morphism $$\X (v) \To \prod_{s(a)= v} \X
    (t(a))$$ induced by $\X (v) \To \X (t(a))$ is a splitting
    epimorphism.
\end{itemize}
\end{definition}

In the proof of the following lemma, we use the fact that  a morphism $\zeta : S_u(R) \to X$ is uniquely determined by defining $\zeta (1_u) \in X(u)$ for any vertex $u$ and representation $X$ of a quiver $Q$ (since $S_u$ is a left adjoint functor of the
evaluation functor $T_v$).
\begin{lemma}\label{prop:For-non-barren-treeTHERE-exists-family-of-inj-rep-directSUM-not-inj.}
If the tree $T$ is not barren, then there exists a family $\{E_i \mid i \in I\}$ of injective representations of $T$ such that $\oplus_{i\in I} E_i$ is not injective, where $I$ is an infinite index set.
\end{lemma}

\begin{proof}
Since $T$ is not barren, there is an infinite set $W$ of vertices of $T$ in such a way that any two distinct vertices are not connected (see \citet[Lemma 3.4]{Enochs-et.al:NoetherianQuivers}). Let us well order $W$, for example $W = \{w_1, w_2, \ldots\}$, and let us consider the representations $S_{w_1}(R)$, $S_{w_2}(R)$, $\ldots$ (where $S_v$ is Mitchell's functor). Then $\sum_{w\in W} S_{w}(R) \subseteq S_v(R)$, where $v$ is the root of the tree, and this sum is direct. Now given $\widetilde{w}\in W$ we have a canonical projection
\[
\xymatrix{S_{\widetilde{w}}(R) \ar[r]^-{\tau_{\widetilde{w}}} &  \oplus_{w\in W} S_w (R) \ar[r]^{\pi_i} & \oplus_{j > i}S_{w_j} (R)}.
\]
And let  us consider the injective hull $\xymatrix{\oplus_{j > i} S_{w_j}(R) \ar@{^{(}->}[r]^-{d_i}   &  E_{w_i}}$. Then we have the following commutative diagram:
\[
\xymatrix{ S_{\widetilde{w}}(R)  \ar@{^{(}->}[r] \ar[d]_{d_i \pi_i \tau_{\widetilde{w}}} &  \oplus_{w\in W} S_w (R) \ar@{-->}[d]^{\varphi} \\
E_{w_i}  \ar@{^{(}->}[r] &  \oplus_{i\geq 1} E_{w_i}}
\]
where $\varphi$ exists by the universal property of the direct sum. So if we assume that $\oplus_{i\geq 1} E_{w_i}$ is injective, then there exists a morphism $\psi: S_v (R) \to \oplus_{i\geq 1} E_{w_i}$ makes the following diagram commute:
\[
\xymatrix{\oplus_{w\in W} S_w (R) \ar@{^{(}->}[r] \ar[d]_{\varphi} &  S_v (R) \ar@{-->}[dl]^{\psi} \\
\oplus_{i\geq 1} E_{w_i}  & }
\]
Now let $w_i \in W$ fixed, but arbitrary. We will show that the $i$th-component  of $\psi (1_v)$ is nonzero. Let $h: \oplus_{i\geq 1} E_{w_i} (v) \to \oplus_{i\geq 1} E_{w_i} (w_{i+1})$ be the morphism corresponding to the representation  $\oplus_{i\geq 1} E_{w_i}$ acting on the path $p$ such that $s(p) = v$ and $t(p)=w_{i+1}$. Then
\[
h \psi (1_v) = \psi (1_{w_{i+1}}) = \varphi (1_{w_{i+1}}) = (\underbrace{\pi_1 (1_{w_{i+1}})}_{\neq 0}, \underbrace{\pi_2 (1_{w_{i+1}})}_{\neq 0}, \ldots, \underbrace{\pi_i (1_{w_{i+1}})}_{\neq 0}, 0, 0, \ldots)
\]
Thus the $i$th-component of $\psi (1_v)$ is nonzero, for all $i\geq 1$,  which is a contradiction. Hence $\oplus_{i\geq 1} E_{w_i}$ cannot be injective.
\end{proof}

\begin{theorem}\label{thm:non-barren-tree-with-starting-finite-arrow-from-vertices-NOT-source-inj}
If the tree $T$ is such that the set $\{t(a): s(a)= v\}$ is
finite for each $v\in T$ and each arrow $a$, and $T$ is not barren, then $T$ is
not a source injective representation quiver.
\end{theorem}

\begin{proof}
Assume that $R$ is a left noetherian ring. Since $T$ is not barren, the previous argument in the proof of Lemma \ref{prop:For-non-barren-treeTHERE-exists-family-of-inj-rep-directSUM-not-inj.} shows that there is a family of injective representations $\{E_w \mid w\in W \}$ such that $\oplus_W E_w$ is not injective. However, as $E_w$ $(w\in W)$ are injective representations, they do satisfy the conditions (i) and (ii) of being a source injective representation quiver (by \citet[Proposition 2.1]{Enochs-et.al:InjectiveRepresentationsOfQuivers}), that is,
\begin{enumerate}
  \item[(i)] $E_w (v)$ is an injective $R$-module for all $v\in T$;
  \item[(ii)] $\displaystyle E_w (v) \To \prod_{s(a)=v} E_w(t(a))$ is a splitting epimorphism.
\end{enumerate}
Therefore, the representation $\oplus_W E_w$ also satisfies (i) and (ii). In fact,
\begin{enumerate}
  \item[(i)] $\displaystyle \big(\bigoplus_W E_w \big)(v) = \bigoplus_W E_w (v) $ is an injective $R$-module, for all $v\in T$ (since $R$ is left noetherian);
  \item[(ii)] $\displaystyle \big(\bigoplus_{w\in W} E_w \big)(v) \To \bigoplus_{w\in W} \prod_{s(a)=v} E_w (t(a)) \cong \prod_{s(a)=v} \big(\bigoplus_{w\in W}E_w\big)(t(a))$ is a splitting epimorphism (where the isomorphism follows since the set $\{t(a) \mid s(a) = v \}$ is finite by hypothesis).
\end{enumerate}
Hence $T$ is not a source injective representation quiver.
\end{proof}

\begin{example}
The binary tree is not a source injective representation quiver.
\end{example}

\begin{remark}
In Theorem \ref{thm:non-barren-tree-with-starting-finite-arrow-from-vertices-NOT-source-inj}, the condition $\{t(a): s(a)= v\}$ is finite for each $v\in T$ cannot be omitted. For instance, if we consider the non-barren tree:
$$
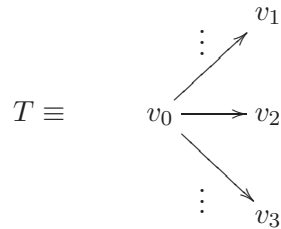
\xymatrix{  &   &  v_1   \\
T \equiv & v_0 \ar[r] \ar[ur]^{\vdots} \ar[dr]_{\vdots}  &  v_2   \\
 &   & v_3 }$$ then the direct sum $\oplus_W E_w$ constructed in Theorem \ref{thm:non-barren-tree-with-starting-finite-arrow-from-vertices-NOT-source-inj} does not satisfy the condition (ii) of being a source injective representation quiver (even if $R$ is left noetherian). Furthermore, $T$ is a source injective representation quiver because it is right rooted quiver (see \citet[Theorem 4.2]{Enochs-et.al:InjectiveRepresentationsOfQuivers}).
\end{remark}

\subsection*{Acknowledgments}
The second author wishes to thank Francisco Guil Asensio for several stimulating
conversations related to the topic of this paper.

\bibliographystyle{deufen}
\bibliography{kaynaklar-Selahattin}
\end{document}